\documentclass[11pt,reqno]{amsproc}

\usepackage[a4paper,margin=1in]{geometry}
\usepackage{amsmath,amssymb,amsthm,mathtools}
\usepackage{enumitem}
\usepackage{hyperref}
\usepackage{mathrsfs}
\usepackage{xcolor}

\hypersetup{
  colorlinks=true,
  linkcolor=blue!60!black,
  citecolor=blue!60!black,
  urlcolor=blue!60!black
}

\newtheorem{theorem}{Theorem}[section]
\newtheorem{proposition}{Proposition}[section]
\newtheorem{lemma}{Lemma}[section]
\newtheorem{corollary}{Corollary}[section]

\newtheorem{definition}{Definition}[section]

\newcommand{\R}{\mathbb R}
\newcommand{\Om}{\Omega}

\newcommand{\one}{\mathbf 1}
\newcommand{\dist}{\operatorname{dist}}
\newcommand{\supp}{\operatorname{supp}}
\newcommand{\dd}{\,d}
\newcommand{\Lam}{\Lambda}

\newcommand{\rev}[1]{#1}

\title[Weak solutions for inviscid SQG]
{Weak Solutions for Inviscid SQG\\
with Lorentz Data}
\date{}

\author{Peter Constantin}
\address{Department of Mathematics, Princeton University, Princeton, NJ 08544}
\email{const@math.princeton.edu}

\author{Mihaela Ignatova}
\address{Department of Mathematics, Temple University, Philadelphia, PA 19122}
\email{ignatova@temple.edu}

\author{Quoc-Hung Nguyen}
\address{Academy of Mathematics and Systems Science - Chinese Academy of Sciences}
\email{qhnguyen@amss.ac.cn} 

\begin{document}

\begin{abstract}
We construct global weak solutions of the inviscid surface quasi-geostrophic
equation in $\mathbb R^2$ and in smooth bounded domains, for arbitrary initial
data in the critical Lorentz space $L^{4/3,2}$.  The solutions conserve the
Hamiltonian $\|\Lambda^{-1/2}\theta(t)\|_{L^2}^2$ for all times.  The second Lorentz exponent
is determined by the sharp boundedness
$\Lambda^{-1/2}:L^{4/3,2}\to L^2$; the corresponding estimate fails for
$L^{4/3,q}$ when $q>2$.  We use an approximation scheme that is tailored for Lorentz spaces. It smooths the
advecting velocity and the initial data, and preserves order in distribution functions.   Removing the approximation is made possible by 
uniform bounds for the Lorentz norms of high amplitude cutoffs of the solutions.

\end{abstract}

\maketitle

\section{Introduction and main result}\label{sec:introduction}

The inviscid surface quasi-geostrophic (SQG) equation is an active scalar model
in which a transported scalar determines the incompressible velocity through a
singular integral of order zero.  It arose as a model for front formation in
geophysical flow and has long served as a two-dimensional analogue of the
three-dimensional incompressible Euler equations 
\cite{Constantin1994, ConstantinMajdaTabak1994}.  The finite time blow up problem from smooth and localized initial data remains a challenging 
open problem.

 In either the plane or in a smooth bounded
domain, the equation takes the form
\[
  \partial_t\theta+u\cdot\nabla\theta=0,
  \qquad u=\nabla^\perp\Lambda^{-1}\theta.
\]
where $\Lambda = (-\Delta)^{\frac{1}{2}}$.  
For smooth solutions, divergence-free transport preserves the distribution
function of $\theta$, and hence every finite $L^p$ norm.  It also
preserves the Hamiltonian
\[
  \mathcal H(\theta)
  =\frac12\int_\Om\theta\Lambda^{-1}\theta\dd x
  =\frac12\|\Lambda^{-1/2}\theta\|_{L^2(\Om)}^2,
\]
because $u=\nabla^\perp\Lambda^{-1}\theta$ is pointwise orthogonal to
$\nabla\Lambda^{-1}\theta$.

Both statements may break down for weak solutions.  The velocity and scalar
may be too rough for the product $u\theta$ to have a direct meaning, while weak
convergence in the natural negative Sobolev space need not preserve the
quadratic Hamiltonian.  Global weak solutions in $L^2$ were first obtained in 
the thesis \cite{Resnick1995}. This was extended to
$L^p$-based settings in
\cite{Marchand2008}; for the equation with homogeneous Dirichlet boundary condition in bounded
domains, global $L^2$ weak solutions and their Hamiltonian conservation were
established in \cite{ConstantinNguyen2018}.  At lower regularity,
convex integration constructions showed that weak SQG solutions need not be
unique \cite{BuckmasterShkollerVicol2019,IsettMa2021}.  More recently, 
in \cite{BrueNguyenJin2026} the authors obtained flexibility and nonuniqueness for weak
solutions on $\mathbb T^2$ with
$\theta\in C([0,1];L^{\bar p})$, where
$\bar p=4/3+10^{-5}$ .  This result is complementary
to the present theorem, which constructs a Hamiltonian-conserving solution for
every initial datum at the Lorentz endpoint.

Two recent results on $\mathbb T^2$ are closely related to this problem.  In \cite{DeRosaLatoccaPark2025, {DeRosaYuzbasioglu2026}} the authors 
constructed Hamiltonian-conserving solutions
for initial data in $L^{4/3}(\mathbb T^2)$ and in Sobolev spaces by vanishing viscosity methods. 

 In this paper we replace $L^{4/3}$ by the larger Lorentz
space $L^{4/3,2}$, consider both the plane and the Dirichlet boundary condition equations on
bounded domains, and regularize the velocity advection law without adding viscosity.
Since $L^{4/3}=L^{4/3,4/3}\subsetneq L^{4/3,2}$, the Lorentz endpoint allows
rougher data than the Lebesgue space.

The natural Lorentz space for the Hamiltonian is given by the
two-dimensional bound
\[
  \|\Lambda^{-1/2}f\|_{L^2}
  \lesssim \|f\|_{L^{4/3,2}},
\]
and the secondary Lorentz exponent $2$ is optimal: the estimate fails with
$L^{4/3,q}$ for every $q>2$.  Thus $L^{4/3,2}$ is the largest Lorentz space at
the fixed principal exponent $4/3$ for which this estimate controls the
Hamiltonian.  We prove the estimate using $L^{p,q}$ convolution inequalities
\cite{ONeil1963}.

Bounded sequences in $L^{4/3,2}$ may concentrate, and weak convergence alone
does not preserve the Hamiltonian.  For our approximations, however, all
amplitude distributions come from heat regularizations of the initial datum.
We use the high amplitude cutoff modulus
$\tau_M(f)=\|(|f|-M)_+\|_{L^{4/3,2}}$.  It is contracted by the heat semigroup,
preserved by measure-preserving transport, and weakly lower semicontinuous.
In the plane, we use also a transport estimate to obtain spatial tightness.

We formulate the notion of weak solution by symmetrizing the nonlinearity. This is an idea that was used before for two dimensional
incompressible Euler equations in vorticity formulation  \cite{Delort1991, Schochet1995}. Like SQG, the 2D vorticity equation is an active scalar with quadratic nonlinearity, 
but it is smoother than SQG whose $\nabla^{\perp}\theta$ is the analogue of the  3D Euler vorticity, The 2D Biot-Savart kernel has a singularity of order 1 near 
the diagonal, and  symmetrization paired with a test function produces a bounded kernel.  For SQG, the kernel has a singularity of order 2, 
\(K_{\rm SQG}(z)=c z^\perp/|z|^3\), and symmetrization paired with a test function produces  a kernel of size
\(|x-y|^{-1}\).  Smallness near
the diagonal is obtained from the $L^{4/3,2}$ high amplitude cutoff  modulus estimate.  In the whole
space, the transport estimate prevents travel to spatial infinity.  

\subsection{The equation and functional setting}\label{subsec:setting}

Let either $\Om=\R^2$ or let $\Om\subset\R^2$ be a bounded smooth domain.
Throughout the paper, implicit constants may depend on the fixed domain
$\Om$ and on explicitly displayed test functions, but not on time or on the
approximation parameters unless stated otherwise.  We set
\[
  \Lam=(-\Delta)^{1/2}.
\]
On $\R^2$ this is the Fourier multiplier, $\Lam = |D|$.  On a bounded domain we use 
spectral calculus of the Laplacian with homogeneous Dirichlet boundary conditions.  Thus
\begin{equation}\label{eq:heat-representation}
  \Lam^{-1} f=(-\Delta)^{-1/2}f
  = c_0\int_0^\infty s^{-1/2}e^{s\Delta}f\,\dd s
\end{equation}
where $e^{s\Delta}$ is the Dirichlet heat semigroup, i.e.,
$w(s,x)=e^{s\Delta}f(x)$ is the solution of
\begin{equation}
	\begin{cases}
		\partial_s w-\Delta w=0,
		& s>0,\quad x\in\Om,\\
		w=0,
		& s >0,\quad x\in\partial\Om,\\
		w(0,x)=f(x),
		& x\in\Om.
	\end{cases}
\end{equation}
We consider the inviscid SQG equation
\begin{equation}\label{eq:sqg}
  \partial_t\theta+u\cdot\nabla\theta=0,
  \qquad
  u=\nabla^\perp\Lam^{-1}\theta .
\end{equation}
The corresponding Hamiltonian energy is
\begin{equation}\label{eq:hminus-energy}
  \|\theta\|_{\dot H^{-1/2}(\Om)}^2
  :=\int_\Om \theta\,\Lam^{-1}\theta\dd x
  =\|\Lam^{-1/2}\theta\|_{L^2(\Om)}^2.
\end{equation}
For a measurable function $f$ we denote its distribution function by
\[
  d_f(\lambda)=|\{x\in\Om:\ |f(x)|>\lambda\}|,
  \qquad \lambda>0,
\]
where $|E|$ is the Lebesgue measure of $E$.
We recall that, for $0<p<\infty$ and $0<q<\infty$,
\begin{equation}\label{eq:lorentz-distribution}
  \|f\|_{L^{p,q}}^q
  \sim
  \int_0^\infty \lambda^{q-1}d_f(\lambda)^{q/p}\dd\lambda .
\end{equation}
In particular,
\begin{equation}\label{eq:l432-distribution}
  \|f\|_{L^{4/3,2}}^2
  \sim
  \int_0^\infty \lambda\,d_f(\lambda)^{3/2}\dd\lambda .
\end{equation}

Let $f^*$ be the decreasing rearrangement of $|f|$ and
$f^{**}(s)=s^{-1}\int_0^s f^*(r)\dd r$.  Following
\cite[Chapter~2]{BennettSharpley1988}, we use on
$X:=L^{4/3,2}(\Om)$ also the rearrangement norm
\begin{equation}\label{eq:fully-symmetric-lorentz-norm}
 \|f\|_X=
 \left(\int_0^{|\Om|}
 [s^{3/4}f^{**}(s)]^2\frac{\dd s}{s}\right)^{1/2},
\end{equation}
where the upper endpoint is infinity when $|\Om|=\infty$.  For $M\ge0$
define the high amplitude cutoff modulus
\begin{equation}\label{eq:tail-modulus}
  \tau_M(f):=\bigl\|(|f|-M)_+\bigr\|_X.
\end{equation}
Because $X$ has order-continuous norm, $\tau_M(f)\to0$ as $M\to\infty$ for
every $f\in X$.  Moreover,
\begin{equation}\label{eq:superlevel-by-modulus}
  \|f\one_{\{|f|>2M\}}\|_X\le2\tau_M(f).
\end{equation}
The functional $\tau_M$ is convex and norm-continuous, and is therefore weakly
lower semicontinuous.

Let $G_1(x,y)$ denote the kernel of $\Lambda^{-1}$ and set
$K_\Om(x,y)=\nabla_x^\perp G_1(x,y)$.  In the whole-space case this reduces to
$K_{\R^2}(x,y)=c(x-y)^\perp/|x-y|^3$.  For a spatial test function $\varphi$,
define the symmetrized transport form
\begin{equation}\label{eq:intro-transport-form}
  \mathcal N_\Om(\theta,\varphi)
  :=\frac12\iint_{\Om\times\Om}
  \left[K_\Om(x,y)\cdot\nabla\varphi(x)
  +K_\Om(y,x)\cdot\nabla\varphi(y)\right]
  \theta(x)\theta(y)\dd x\dd y.
\end{equation}
The cancellation displayed in Section~\ref{sec:symmetrized} makes this form
finite for $\theta\in L^{4/3,2}$.

\begin{definition}[Weak solution]\label{def:weak-solution}
Let $\theta_0\in L^{4/3,2}(\Om)$.  A function
\[
  \theta\in L^\infty_{\mathrm{loc}}
  ([0,\infty);L^{4/3,2}(\Om))
  \cap C([0,\infty);\mathcal D'(\Om))
\]
is a weak solution of \eqref{eq:sqg} with initial datum $\theta_0$ if, for every
$\varphi\in C_c^\infty([0,\infty)\times\Om)$,
\begin{equation}\label{eq:weak-solution-definition}
  \int_0^\infty\!\int_\Om
  \theta\,\partial_t\varphi\dd x\dd t
  +\int_\Om\theta_0(x)\varphi(0,x)\dd x
  +\int_0^\infty\mathcal N_\Om(\theta(t),\varphi(t))\dd t=0.
\end{equation}
\end{definition}

\subsection{Main theorem}\label{subsec:main-theorem}

\begin{theorem}[Weak solutions at critical Lorentz regularity]\label{thm:main}
Let $\theta_0\in L^{4/3,2}(\Om)$.
Then there exists a global weak solution $\theta$ of \eqref{eq:sqg} such that
\[
  \theta\in C([0,\infty);X_{\mathrm w})
  \cap L^\infty_{\mathrm{loc}}([0,\infty);X).
\]
Moreover, for every $M\ge0$ and every $t\ge0$,
\begin{equation}\label{eq:tail-bound}
  \tau_M(\theta(t))\le\tau_M(\theta_0).
\end{equation}
There also exists a nonnegative nondecreasing function
$\Phi=\Phi_{\theta_0}$, depending only on the initial datum, such that
\[
  \Phi(\lambda)\longrightarrow\infty
  \qquad\text{as }\lambda\longrightarrow\infty,
  \qquad
  1+\Phi(2\lambda)\le2\bigl(1+\Phi(\lambda)\bigr),
\]
and, for every $t\ge0$,
\begin{equation}\label{eq:weighted-distribution-bound}
  \int_0^\infty
  \lambda\bigl(1+\Phi(\lambda)\bigr)
  d_{\theta(t)}(\lambda)^{3/2}\dd\lambda
  \lesssim
  \int_0^\infty \lambda d_{\theta_0}(\lambda)^{3/2}\dd\lambda
  \simeq \|\theta_0\|_X^2.
\end{equation}
Finally, the negative Sobolev energy is conserved:
\begin{equation}\label{eq:hminus-conservation}
  \|\theta(t)\|_{\dot H^{-1/2}(\Om)}
  =
  \|\theta_0\|_{\dot H^{-1/2}(\Om)}
  \qquad\text{for every }t\ge0 .
\end{equation}
\end{theorem}

\subsection{Main ideas and organization}\label{subsec:strategy}

We truncate the heat semigroup representation of $\Lambda^{-1}$ and smooth
the initial datum.  The resulting approximate equation is then a transport equation
with a smooth divergence-free velocity. It has global smooth solutions. This approximation was used in \cite{Ignatova2019}
to construct global regular solutions of critical dissipative SQG in bounded domains. In the present case without dissipation
 the distribution functions of the solutions are preserved in time, and the self-adjointness of the regularized operator gives an exact
Hamiltonian identity. 

The main difficulty is to prevent concentration on the diagonal.  One of the
key compactness ingredients is given by the existence of the uniform factor $\Phi$ 
of Proposition~\ref{prop:weighted-tail-envelope} which, together with the 
high amplitude cutoff modulus $\tau_M$ allows passage to the limit.

Section~\ref{sec:lorentz-tail} contains the endpoint Lorentz and high amplitude cutoff 
modulus estimates.  Section~\ref{sec:approximation} constructs the conservative
approximation and establishes its uniform bounds and Hamiltonian identity.
The symmetrized formulation in the plane and on bounded domains is developed
in Section~\ref{sec:symmetrized}.  Section~\ref{sec:compactness-proof} contains
the compactness argument and the proof of Theorem~\ref{thm:main}, while
Appendix~\ref{app:sharpness} proves the optimality of the secondary Lorentz
exponent.

\section{Critical Lorentz and high amplitude cutoff estimates}\label{sec:lorentz-tail}

We first prove the Lorentz estimate for the Hamiltonian and then collect the
high amplitude cutoff modulus estimates used in the compactness argument.  The relevance of the pair
$(4/3,2)$ can already be read from the kernel: in two dimensions the operator
o $\Lambda^{-1/2}$ has kernel of size $|x|^{-3/2}$, which belongs to weak
$L^{4/3}$.  An $L^{p,q}$ convolution inequality then places the image in
$L^{2,2}=L^2$ when the secondary exponent of the datum is at most $2$.

\begin{proposition}[Endpoint Lorentz control]\label{prop:critical-lorentz}
Let either $\Om=\R^2$ or let $\Om\subset\R^2$ be a bounded smooth domain, and
define
\[
  \mathcal I_{\Om,1/2}f
  :=
  \begin{cases}
    |D|^{-1/2}f,&\Om=\R^2,\\
    \Lam^{-1/2}f,&\Om\text{ bounded}.
  \end{cases}
\]
Then
\begin{equation}\label{eq:critical-lorentz}
  \|\mathcal I_{\Om,1/2}f\|_{L^2(\Om)}
  \lesssim
  \|f\|_{L^{4/3,2}(\Om)}
  \lesssim
  \left(
  \int_0^\infty \lambda d_f(\lambda)^{3/2}\dd\lambda
  \right)^{1/2}.
\end{equation}
Furthermore, on $\R^2$ the estimate
\[
  \| |D|^{-1/2}f\|_{L^2}\lesssim \|f\|_{L^{4/3,q}}
\]
fails in general for every $q>2$.
\end{proposition}

\begin{proof}
We first consider $\Omega=\mathbb R^2$.  The Riesz-potential representation is
\[
 |D|^{-1/2}f=I_{1/2}f=c_2|x|^{-3/2}*f.
\]
Since
\[
 \bigl|\{x:|x|^{-3/2}>\lambda\}\bigr|\sim\lambda^{-4/3},
\]
the kernel $|x|^{-3/2}$ belongs to $L^{4/3,\infty}(\mathbb R^2)$.
The convolution inequality \cite{ONeil1963}
\[
\|f*g\|_{L^{r,c}} \le \|f\|_{L^{p, a}}\|g\|_{L^{q,b}},
\]
where $r^{-1} = p^{-1}+ q^{-1} - 1$, for $p,q\in (1,\infty)$,  $c^{-1} = a^{-1}+ b^{-1}$, for $a,b\in [1,\infty]$,
therefore gives 
\[
 \bigl\||D|^{-1/2}f\bigr\|_{L^2}
 \lesssim\|f\|_{L^{4/3,2}},
\]
because $L^{4/3,\infty}*L^{4/3,2}\subset L^{2,2}=L^2$.  The second
inequality in \eqref{eq:critical-lorentz} follows from
\eqref{eq:l432-distribution}.
	
For a bounded smooth domain, the same mapping property follows from heat-kernel
domination.  The parabolic maximum principle gives the comparison
\[
  0\leq H_D(\rho,x,y)
  \leq(4\pi\rho)^{-1}
  \exp\left(-\frac{|x-y|^2}{4\rho}\right)
\]
(see, for example,
\cite{ConstantinIgnatova2017,ConstantinNguyen2018}).  The fractional
representation
\[
  \Lambda^{-1/2}f
  =c\int_0^\infty e^{\rho\Delta}f\,\rho^{-3/4}\dd\rho
\]
therefore has an integral kernel bounded by
\[
 C\int_0^\infty \rho^{-7/4}
 \exp\left(-\frac{|x-y|^2}{4\rho}\right)\dd\rho
 =C|x-y|^{-3/2},
\]
uniformly for $x,y\in\Omega$.  Extending $f$ by zero outside $\Omega$ and applying the
preceding Lorentz convolution estimate gives
\[
  \|\Lambda^{-1/2}f\|_{L^2(\Omega)}
  \lesssim \|f\|_{L^{4/3,2}(\Omega)}.
\]
This proves the upper estimate in both geometries.  The failure for $q>2$ is
proved in Appendix~\ref{app:sharpness}.
\end{proof}

\subsection{High amplitude cutoff  estimate}\label{subsec:tail-estimate}

We next collect the properties of $\tau_M$ used below.

\begin{lemma}\label{lem:tail}
Let $X=L^{4/3,2}(\Om)$ be equipped with the norm
\eqref{eq:fully-symmetric-lorentz-norm}.  Then, for every $M\geq0$:
\begin{enumerate}[label=\textup{(\roman*)}]
\item $f\mapsto\tau_M(f)$ is convex and norm-continuous, hence weakly lower
semicontinuous on $X$;
\item if $T$ is a positive linear operator satisfying
\[
 \|Tg\|_{L^1(\Om)}\leq\|g\|_{L^1(\Om)},
 \qquad
 \|Tg\|_{L^\infty(\Om)}\leq\|g\|_{L^\infty(\Om)},
\]
then $\tau_M(Tf)\leq\tau_M(f)$;
\item if $Y$ is measure preserving, then
$\tau_M(f\circ Y)=\tau_M(f)$;
\item with
$T_Mf=\operatorname{sgn}(f)(|f|\wedge M)$ and
$R_Mf=f-T_Mf=\operatorname{sgn}(f)(|f|-M)_+$,
\begin{equation}\label{eq:tail-estimate}
 \|R_Mf\|_X=\tau_M(f),
 \qquad \|\Lam^{-1/2}R_Mf\|_{L^2}\lesssim\tau_M(f).
\end{equation}
Moreover, the restriction of $f$ to the superlevel set
$\{|f|>2M\}$ satisfies
\begin{equation}
 \|f\one_{\{|f|>2M\}}\|_X\leq2\tau_M(f),
 \qquad
 \|\Lam^{-1/2}(f\one_{\{|f|>2M\}})\|_{L^2}
 \lesssim\tau_M(f).
 \label{eq:hard-tail-by-modulus}
\end{equation}
\end{enumerate}
In particular, $\tau_M(f)\to0$ as $M\to\infty$ for every $f\in X$.
\end{lemma}

\begin{proof}
The scalar function $\vartheta_M(s)=(|s|-M)_+$ is convex and $1$-Lipschitz.
In particular,
\[
 |\vartheta_M(f)-\vartheta_M(g)|\le |f-g|
 \quad\hbox{a.e.},
\]
so the lattice property gives
$|\tau_M(f)-\tau_M(g)|\le\|f-g\|_X$.  The same lattice property and the
pointwise convexity of $\vartheta_M$ make $\tau_M$ convex; hence it is
norm-continuous and
every continuous convex functional on a Banach space is weakly lower
semicontinuous.  Positivity gives $|Tf|\le T|f|$.  Since $T$ is positive and
$T\one\leq\one$, Jensen's inequality, with the missing mass placed at zero,
gives
\[
 (|Tf|-M)_+\le T\bigl((|f|-M)_+\bigr).
\]
Set $h=(|f|-M)_+$.  The $L^1$ and $L^\infty$ contraction assumptions imply
$Th\prec\!\prec h$.  Here $g\prec\!\prec h$ means that $g$ is
submajorized by $h$, that is,
\[
  \int_0^s g^*(r)\dd r
  \leq
  \int_0^s h^*(r)\dd r
  \qquad\text{for every }0<s<|\Om|,
\]
with the condition required for every $s>0$ when $|\Om|=\infty$.
Monotonicity of the norm
\eqref{eq:fully-symmetric-lorentz-norm} under submajorization therefore gives
\[
 \tau_M(Tf)\leq\|Th\|_X\leq\|h\|_X=\tau_M(f),
\]
which proves (ii).  Assertion (iii)
follows from equimeasurability.  The identity $|R_Mf|=(|f|-M)_+$ and
Proposition~\ref{prop:critical-lorentz} give \eqref{eq:tail-estimate}.  On
$\{|f|>2M\}$, the inequality $|f|\le2(|f|-M)_+$ and the lattice property
give the stated $X$ estimate.  Applying Proposition~\ref{prop:critical-lorentz}
to $f\one_{\{|f|>2M\}}$ gives the corresponding
$\dot H^{-1/2}$ estimate.
Order continuity of $X$ implies $\tau_M(f)\downarrow0$.
\end{proof}

\begin{proposition}
\label{prop:weighted-tail-envelope}
Let $f\in X=L^{4/3,2}(\Om)$ and define
\[
  \mathcal A_f
  :=\{h\in X:\ \tau_M(h)\le\tau_M(f)
       \text{ for every }M\ge0\}.
\]
Then there exists a nonnegative nondecreasing function
$\Phi_f:[0,\infty)\to[0,\infty)$ such that
\[
  \Phi_f(\lambda)\longrightarrow\infty,
  \qquad
  1+\Phi_f(2\lambda)
  \le2\bigl(1+\Phi_f(\lambda)\bigr).
\]
In particular, the weight $1+\Phi_f$ is doubling.  With
\[
  \mathcal M_{\Phi_f}(h)
  :=\int_0^\infty
  \lambda\bigl(1+\Phi_f(\lambda)\bigr)
  d_h(\lambda)^{3/2}\dd\lambda,
\]
it holds that
\begin{equation}\label{eq:weighted-tail-envelope}
  \sup_{h\in\mathcal A_f}\mathcal M_{\Phi_f}(h)
  \lesssim \mathcal M_{\Phi_f}(f)<\infty.
\end{equation}
Moreover, we normalize $\Phi_f$  so that
\begin{equation}\label{eq:weighted-tail-normalization}
  \mathcal M_{\Phi_f}(f)
  \le2\int_0^\infty \lambda d_f(\lambda)^{3/2}\dd\lambda.
\end{equation}
\end{proposition}

\begin{proof}
If $f=0$, then $\mathcal A_f=\{0\}$ and one may take
$\Phi_f(\lambda)=\lambda$, so assume $f\ne0$ and put
\[
  I_f:=\int_0^\infty \lambda d_f(\lambda)^{3/2}\dd\lambda>0.
\]
For $h\in\mathcal A_f$, the distribution-function representation of the
Lorentz norm and \eqref{eq:hard-tail-by-modulus} give
\begin{align*}
  \int_{2M}^\infty \lambda d_h(\lambda)^{3/2}\dd\lambda
  &\le
  \int_0^\infty
  \lambda d_{h\one_{\{|h|>2M\}}}(\lambda)^{3/2}\dd\lambda \\
  &\lesssim
  \|h\one_{\{|h|>2M\}}\|_X^2
  \lesssim \tau_M(f)^2.
\end{align*}
Since $\tau_M(f)\to0$ as $M\to\infty$, the preceding estimate yields
\[
  \sup_{h\in\mathcal A_f}
  \int_{2M}^\infty
  \lambda d_h(\lambda)^{3/2}\dd\lambda
  \lesssim \tau_M(f)^2
  \longrightarrow0.
\]
We may therefore choose an increasing sequence $R_k\to\infty$ with
$R_{k+1}\ge2R_k$ such that
\[
  \sup_{h\in\mathcal A_f}
  \int_{R_k}^\infty \lambda d_h(\lambda)^{3/2}\dd\lambda
  \le2^{-k}I_f,
  \qquad k\ge1.
\]
Set
\[
  \Phi_f(\lambda)
  :=\sum_{k=1}^\infty\one_{[R_k,\infty)}(\lambda).
\]
This function is nondecreasing and tends to infinity.  Because
$R_{k+1}\ge2R_k$, the interval $(\lambda,2\lambda]$ contains at most one new
level $R_k$; hence $1+\Phi_f$ satisfies the displayed doubling estimate.
\rev{Taking $M=0$ in the preceding tail estimate and using
$\tau_0(f)=\|f\|_X$ gives the uniform unweighted bound}
\[
 \rev{\sup_{h\in\mathcal A_f}
 \int_0^\infty\lambda d_h(\lambda)^{3/2}\dd\lambda
 \lesssim \|f\|_X^2\simeq I_f.}
\]
Finally, Tonelli's theorem yields, uniformly for $h\in\mathcal A_f$,
\[
  \int_0^\infty \lambda\Phi_f(\lambda)
  d_h(\lambda)^{3/2}\dd\lambda
  =\sum_{k=1}^\infty
  \int_{R_k}^\infty \lambda d_h(\lambda)^{3/2}\dd\lambda
  \le I_f.
\]
\rev{Together with the preceding unweighted estimate, this proves
\eqref{eq:weighted-tail-envelope}.  For $h=f$, the unweighted term equals
$I_f$, while the term containing $\Phi_f$ is at most $I_f$; hence
\eqref{eq:weighted-tail-normalization} follows with the stated constant.}
\end{proof}

Proposition~\ref{prop:weighted-tail-envelope} gives the weighted estimate in
Theorem~\ref{thm:main}.

\section{Conservative approximation and uniform estimates}
\label{sec:approximation}

We regularize the velocity law by truncating the heat representation of $\Lambda^{-1}$ and
we independently smooth the initial datum.  The first operation gives a smooth
velocity, while self-adjointness of the truncated operator preserves the
Hamiltonian cancellation.

\subsection{Approximation and preservation of distributions}

Let $\epsilon_1,\epsilon_2>0$ and define the regularized $\Lambda^{-1}$  operator
\begin{equation}\label{eq:regularized-lambda-inverse}
  \Lam_{\epsilon_1}^{-1}f
  :=c_0\int_{\epsilon_1}^\infty e^{\rho\Delta}f\,\rho^{-1/2}\dd\rho .
\end{equation}

 The operator $\Lam_{\epsilon_1}^{-1}$ is self-adjoint and
nonnegative.  As $\epsilon_1\downarrow0$, its quadratic form increases to that
of $\Lam^{-1}$ by monotone convergence in
\eqref{eq:regularized-lambda-inverse}.  The parameter $\epsilon_2$ is used only
to smooth the initial datum.

We solve
\begin{equation}\label{eq:approximation}
  \begin{cases}
  \partial_t\theta^{\epsilon_1,\epsilon_2}
  +u^{\epsilon_1,\epsilon_2}\cdot\nabla\theta^{\epsilon_1,\epsilon_2}=0,\\[2mm]
  u^{\epsilon_1,\epsilon_2}
  =\nabla^\perp\Lam_{\epsilon_1}^{-1}\theta^{\epsilon_1,\epsilon_2},\\[2mm]
  \theta^{\epsilon_1,\epsilon_2}(0)=e^{\epsilon_2\Delta}\theta_0.
  \end{cases}
\end{equation}
\begingroup
For fixed $\epsilon_1,
\epsilon_2>0$, the initial datum belongs to $X\cap L^\infty$ and is smooth.
On $\R^2$, if $K_{\epsilon_1}$ denotes the velocity kernel, then
\[
  K_{\epsilon_1}\in L^{4,2}(\R^2),
  \qquad
  \|\nabla K_{\epsilon_1}\|_{L^1(\R^2)}\leq C_{\epsilon_1}.
\]
Lorentz H\"older and Young inequalities consequently give
\[
 \|u^{\epsilon_1,\epsilon_2}(t)\|_{L^\infty}
 \leq C_{\epsilon_1}\|\theta^{\epsilon_1,\epsilon_2}(t)\|_X,
 \qquad
 \|\nabla u^{\epsilon_1,\epsilon_2}(t)\|_{L^\infty}
 \leq C_{\epsilon_1}
 \|\theta^{\epsilon_1,\epsilon_2}(t)\|_{L^\infty}.
\]
The corresponding bounded-domain estimates follow from the differentiated
Dirichlet heat-kernel representation.  Higher derivatives satisfy analogous
bounds, with constants depending on $\epsilon_1$ and the derivative order.
The velocity is therefore smooth, bounded, and divergence-free.  In a bounded
domain, $\Lam_{\epsilon_1}^{-1}\theta$ has zero Dirichlet trace, so
$u\cdot n=0$ on $\partial\Om$.  Conservation of the $X$ and $L^\infty$ norms
along characteristics then gives global continuation of the classical
solution.
\endgroup

More precisely, if
\[
  \frac{\dd}{\dd t}X^{\epsilon_1,\epsilon_2}(t,a)
  =u^{\epsilon_1,\epsilon_2}
  \bigl(t,X^{\epsilon_1,\epsilon_2}(t,a)\bigr),
  \qquad X^{\epsilon_1,\epsilon_2}(0,a)=a,
\]
then incompressibility gives
$\det D_aX^{\epsilon_1,\epsilon_2}(t,a)=1$ and
\[
  \theta^{\epsilon_1,\epsilon_2}
  \bigl(t,X^{\epsilon_1,\epsilon_2}(t,a)\bigr)
  =e^{\epsilon_2\Delta}\theta_0(a).
\]
Consequently, for every nonnegative Borel function
$\beta:\R\to[0,\infty]$,
\begin{equation}\label{eq:renormalized-approximation}
  \int_\Om \beta(\theta^{\epsilon_1,\epsilon_2}(t,x))\dd x
  =
  \int_\Om \beta(e^{\epsilon_2\Delta}\theta_0(x))\dd x.
\end{equation}
The same identity holds for a signed Borel function whenever
$\beta(e^{\epsilon_2\Delta}\theta_0)$ is integrable.
Equivalently,
\begin{equation}\label{eq:distribution-approximation}
  d_{\theta^{\epsilon_1,\epsilon_2}(t)}(\lambda)
  =d_{e^{\epsilon_2\Delta}\theta_0}(\lambda).
\end{equation}
Thus all amplitude information at positive times is inherited from the single
smoothed function $e^{\epsilon_2\Delta}\theta_0$; no new high-amplitude tail is
created by the approximate dynamics.  Lemma~\ref{lem:tail}, applied first to
the sub-Markov heat semigroup and then to the measure-preserving flow, gives
for every $M\ge0$
\begin{equation}\label{eq:uniform-tail}
  \sup_{\epsilon_1,\epsilon_2>0}\sup_{t\ge0}
  \tau_M(\theta^{\epsilon_1,\epsilon_2}(t))
  \le \tau_M(\theta_0).
\end{equation}
Taking $M=0$ gives
\begin{equation}\label{eq:uniform-lorentz}
  \sup_{\epsilon_1,\epsilon_2>0}\sup_{t\ge0}
  \|\theta^{\epsilon_1,\epsilon_2}(t)\|_X
  \le \|\theta_0\|_X,
\end{equation}
and hence, by Proposition~\ref{prop:critical-lorentz},
\begin{equation}\label{eq:uniform-hminus}
  \sup_{\epsilon_1,\epsilon_2>0}\sup_{t\ge0}
  \|\theta^{\epsilon_1,\epsilon_2}(t)\|_{\dot H^{-1/2}}
  \lesssim \|\theta_0\|_X.
\end{equation}
The constants in \eqref{eq:uniform-tail}--\eqref{eq:uniform-hminus} are
independent of time and of both regularization parameters.  This uniformity is
what permits a diagonal limit on every compact time interval.

\begin{lemma}\label{lem:whole-space-tightness}
Let $\Om=\R^2$ and  $0<\epsilon_1,\epsilon_2\le1$.  For every
$T>0$,
\begin{equation}\label{eq:whole-space-tightness}
 \lim_{R\to\infty}\ \sup_{0<\epsilon_1,\epsilon_2\le1}
 \sup_{0\le t\le T}
 \|\theta^{\epsilon_1,\epsilon_2}(t)\one_{\{|x|>R\}}\|_{L^{4/3,2}}=0.
\end{equation}
\end{lemma}

\begin{proof}
We write $X=L^{4/3,2}(\R^2)$ and abbreviate
$\theta^\epsilon=\theta^{\epsilon_1,\epsilon_2}$,
$u^\epsilon=u^{\epsilon_1,\epsilon_2}$, and
$f^{\epsilon_2}=e^{\epsilon_2\Delta}\theta_0$.  The regularized Riesz
transforms are uniformly bounded on $X$, so \eqref{eq:uniform-lorentz} gives
\begin{equation}\label{eq:uniform-velocity-lorentz}
 \sup_{\epsilon_1,\epsilon_2,t}
 \|u^{\epsilon_1,\epsilon_2}(t)\|_X\le U.
\end{equation}
Indeed, their multipliers are the Riesz multipliers times
$q_\epsilon(r)=c\int_{\epsilon r^2}^\infty e^{-s}s^{-1/2}\dd s$.  For
$j=0,1,2$, we have
\[
  \sup_{\epsilon>0}\sup_{r>0}
  r^j|\partial_r^jq_\epsilon(r)|\leq C_j.
\]
The Mikhlin multiplier theorem and the fact that $L^{p,q}$ is an interpolation space give \eqref{eq:uniform-velocity-lorentz}.
Strong continuity of the heat semigroup on $X$ makes
$\{e^{s\Delta}\theta_0:0\le s\le1\}$ compact in $X$.  Since $X$ has
order-continuous norm,
\begin{equation}\label{eq:initial-spatial-tail}
 \alpha(R_0):=\sup_{0\le s\le1}
 \|e^{s\Delta}\theta_0\one_{\{|x|>R_0\}}\|_X\longrightarrow0.
\end{equation}
Indeed, multiplication by $\one_{\{|x|>R_0\}}$ is uniformly bounded on $X$
and converges strongly to zero as $R_0\to\infty$; this convergence is uniform
on compact subsets of $X$ by a finite-net argument.
Let $X^\epsilon(t,a)$ be the measure-preserving flow.  Rearrangement
invariance and the flow formula yield
\[
 \|\theta^\epsilon(t)\one_{\{|x|>R\}}\|_X
 =\|f^{\epsilon_2}(a)
   \one_{\{|X^\epsilon(t,a)|>R\}}\|_X.
\]
We fix $M,R_0>0$ and split the last function according to
$|a|>R_0$, $|f^{\epsilon_2}(a)|>2M$, and the remaining set
\[
 E=\{a:|a|\le R_0,\ |f^{\epsilon_2}(a)|\le2M,
              |X^\epsilon(t,a)|>R\}.
\]
The first two pieces are bounded by $\alpha(R_0)$ and
$2\tau_M(\theta_0)$, respectively, using
\eqref{eq:uniform-tail}.  Now every $a\in E$ has moved a distance of at least
$R-R_0$.  Measure preservation, Lorentz H\"older, and
\eqref{eq:uniform-velocity-lorentz} give
\begin{align*}
 (R-R_0)|E|
 &\le\int_E|X^\epsilon(t,a)-a|\dd a
 \le\int_0^t\int_{X^\epsilon(s,E)}|u^\epsilon(s,x)|\dd x\dd s \le C T U |E|^{1/4}.
\end{align*}
Here $X^\epsilon(s,E)=\{X^\epsilon(s,a):a\in E\}$; it has measure $|E|$
because the flow is measure preserving. We used the duality $(L^{\frac{4}{3},2})' = L^{4,2}$
and the fact that for any measurable set
\[
\|\one_A\|_{L^{p,q}} = | A|^{\frac{1}{p}}
\]
holds for $1<p<\infty$, $1\le q\le\infty$. Thus $|E|^{3/4}\le CTU/D$.  Since
$\|\one_E\|_X\simeq|E|^{3/4}$, the remaining piece has norm at most
$CMTU/(R-R_0)$.  We have proved
\begin{equation}\label{eq:tightness-quantitative}
 \sup_{0<\epsilon_1,\epsilon_2\le1}\sup_{0\le t\le T}
 \|\theta^\epsilon(t)\one_{\{|x|>R\}}\|_X
 \le \alpha(R_0)+2\tau_M(\theta_0)+\frac{CMTU}{R-R_0}.
\end{equation}
We first choose $M$ large, then $R_0$ large, and finally $R$ large.  This proves
\eqref{eq:whole-space-tightness}.
\end{proof}

The preceding estimate also controls the far field.  If a
bilinear kernel is bounded by $C|x-y|^{-1}$, the contribution in which either
variable lies outside $B_R$ is at most
\[
 C\|\theta\one_{B_R^c}\|_{L^{4/3,2}}
  \|\theta\|_{L^{4/3,2}}
\]
by Proposition~\ref{prop:critical-lorentz}.  Thus both the transport form and
the Hamiltonian can be localized to $B_R\times B_R$ before the diagonal is
removed.

\subsection{The regularized Hamiltonian identity}

For $f\in X$, set
\[
  \mathcal H_{\epsilon_1}(f)
  :=\frac12\int_\Om f\Lam_{\epsilon_1}^{-1}f\dd x.
\]
This is well defined by Proposition~\ref{prop:critical-lorentz}.  Moreover,
\[
 \langle f,\Lam_{\epsilon_1}^{-1}g\rangle
 =c_0\int_{\epsilon_1}^{\infty}
 \langle e^{\rho\Delta/2}f,e^{\rho\Delta/2}g\rangle
 \rho^{-1/2}\dd\rho,
\]
so the form is symmetric and nonnegative.  Write
\[
 \theta^\epsilon=\theta^{\epsilon_1,\epsilon_2},
 \qquad
 \psi^\epsilon=\Lam_{\epsilon_1}^{-1}\theta^\epsilon,
 \qquad
 u^\epsilon=\nabla^\perp\psi^\epsilon.
\]
Testing the approximate equation with $\psi^\epsilon$ and using
self-adjointness gives
\[
 \frac{d}{dt}\mathcal H_{\epsilon_1}(\theta^\epsilon)
 =-\int_\Om u^\epsilon\cdot\nabla\theta^\epsilon\,\psi^\epsilon\dd x
 =\int_\Om\theta^\epsilon
 \nabla^\perp\psi^\epsilon\cdot\nabla\psi^\epsilon\dd x=0.
\]
On a bounded domain there is no boundary term because
$u^\epsilon\cdot n=0$ on $\partial\Om$.  On $\R^2$, the same calculation
follows by inserting a cutoff: for fixed
$\epsilon_1,\epsilon_2>0$, one has
$\theta^\epsilon,u^\epsilon\in L^2$ and $\psi^\epsilon\in L^\infty$, so the
cutoff error tends to zero.  Therefore
\begin{equation}\label{eq:regularized-energy}
 \int_\Omega\theta^{\epsilon_1,\epsilon_2}(t)
 \Lam_{\epsilon_1}^{-1}\theta^{\epsilon_1,\epsilon_2}(t)\dd x
 =\int_\Omega e^{\epsilon_2\Delta}\theta_0
 \Lam_{\epsilon_1}^{-1}e^{\epsilon_2\Delta}\theta_0\dd x.
\end{equation}

Finally,
\begin{equation}\label{eq:initial-regularized-energy-convergence}
 \int_\Om e^{\epsilon_2\Delta}\theta_0
 \Lam_{\epsilon_1}^{-1}e^{\epsilon_2\Delta}\theta_0\dd x
 \longrightarrow
 \int_\Om\theta_0\Lam^{-1}\theta_0\dd x
 \qquad (\epsilon_1,\epsilon_2\downarrow0).
\end{equation}
Indeed, $e^{\epsilon_2\Delta}\theta_0\to\theta_0$ in
$\dot H^{-1/2}(\Om)$, while the nonnegative quadratic forms associated with
$\Lam_{\epsilon_1}^{-1}$ are bounded by the $\Lam^{-1}$ form and increase to
it.  This proves \eqref{eq:initial-regularized-energy-convergence}.

\section{The symmetrized weak formulation}\label{sec:symmetrized}

At critical Lorentz regularity the product $u\theta$ need not be defined.
After symmetrization, however, the order-two velocity kernel is paired with a
difference of test-function gradients.  The resulting kernel is bounded by
$C|x-y|^{-1}$ and defines a continuous quadratic form on $L^{4/3,2}$.  We
prove this first in the plane and then on bounded domains.

\subsection{Whole space}

On $\R^2$,
\[
  u(x)=K*\theta(x),
  \qquad
  K(x)=c\frac{x^\perp}{|x|^3}.
\]
We first perform the calculation for smooth $\theta$ with sufficient decay.
The estimate obtained below then defines the quadratic term for general
$L^{4/3,2}$ data by density and continuity.
Let $\varphi\in C_c^\infty(\R^2)$.  Integrating the transport term by parts
and inserting the kernel representation of $u$ gives
\[
  \int_{\R^2}u\cdot\nabla\theta\,\varphi\dd x
  =-
  \iint_{\R^2\times\R^2}
  K(x-y)\cdot\nabla\varphi(x)\theta(x)\theta(y)\dd x\dd y.
\]
We now interchange $x$ and $y$ in the double integral and average the two
representations.  Since $K(y-x)=-K(x-y)$, this produces
\begin{equation}\label{eq:whole-space-symmetrized}
  \int_{\R^2}u\cdot\nabla\theta\,\varphi\dd x
  =-\frac12
  \iint_{\R^2\times\R^2}
  K(x-y)\cdot(\nabla\varphi(x)-\nabla\varphi(y))
  \theta(x)\theta(y)\dd x\dd y.
\end{equation}
Hence the singularity of the symmetrized kernel is only of order one:
\begin{equation}\label{eq:whole-space-nonlinear-bound}
  \left|\int_{\R^2}u\cdot\nabla\theta\,\varphi\dd x\right|
  \lesssim
  \|\nabla^2\varphi\|_{L^\infty}
  \iint_{\R^2\times\R^2}
  \frac{|\theta(x)||\theta(y)|}{|x-y|}\dd x\dd y.
\end{equation}
The remaining double integral is the Riesz-potential quadratic form of
$|\theta|$ and is controlled by Proposition~\ref{prop:critical-lorentz}:
\[
  \iint\frac{|\theta(x)||\theta(y)|}{|x-y|}\dd x\dd y
  \lesssim
  \||D|^{-1/2}|\theta|\|_{L^2}^2
  \lesssim
  \|\theta\|_{L^{4/3,2}}^2.
\]
Therefore
\begin{equation}\label{eq:time-bound-whole}
  \left|\int_{\R^2}u\cdot\nabla\theta\,\varphi\dd x\right|
  \lesssim
  \|\nabla^2\varphi\|_{L^\infty}
  \|\theta\|_{L^{4/3,2}}^2.
\end{equation}
More generally, the same calculation and
Cauchy--Schwarz for Riesz potentials give
\[
\begin{aligned}
  \left|
  \iint_{\R^2\times\R^2}
  K(x-y)\cdot(\nabla\varphi(x)-\nabla\varphi(y))
  f(x)g(y)\dd x\dd y
  \right| \lesssim
  \|\nabla^2\varphi\|_{L^\infty}
  \|f\|_{L^{4/3,2}}\|g\|_{L^{4/3,2}}.
\end{aligned}
\]
Consequently, the right-hand side of
\eqref{eq:whole-space-symmetrized} extends continuously from smooth functions
to $L^{4/3,2}$.  We use this extension as the nonlinear term.

We use the corresponding estimate for the truncated kernel.

\begin{lemma}
\label{lem:regularized-whole-space-kernel}
Let
\[
 H(\rho,z)=\frac1{4\pi\rho}e^{-|z|^2/(4\rho)},\qquad
 G_\epsilon(z)=c_0\int_\epsilon^\infty
 H(\rho,z)\rho^{-1/2}\dd\rho,\qquad
 K_\epsilon=\nabla^\perp G_\epsilon
\]
for $\epsilon\ge0$ and $z\ne0$.  Then $G_\epsilon$ is even,
$K_\epsilon$ is odd, and, uniformly for $\epsilon\ge0$,
\begin{equation}\label{eq:regularized-whole-space-bounds}
 |G_\epsilon(z)|\le \frac{C}{|z|},\qquad
 |K_\epsilon(z)|\le \frac{C}{|z|^2}.
\end{equation}
Consequently, for every $\varphi\in C_c^\infty(\R^2)$,
\begin{equation}\label{eq:regularized-whole-space-paired}
 \left|K_\epsilon(x-y)\cdot
 \bigl(\nabla\varphi(x)-\nabla\varphi(y)\bigr)\right|
 \le \frac{C\|\nabla^2\varphi\|_{L^\infty}}{|x-y|}.
\end{equation}
For every $\eta>0$,
\begin{equation}\label{eq:whole-space-off-diagonal-convergence}
 \sup_{|z|\ge\eta}
 \bigl(|G_\epsilon(z)-G_0(z)|
       +|K_\epsilon(z)-K_0(z)|\bigr)\longrightarrow0
 \qquad(\epsilon\downarrow0).
\end{equation}
\end{lemma}

\begin{proof}
Parity follows from the radial heat kernel.  The Gaussian estimates
\[
 |H(\rho,z)|\le C\rho^{-1}e^{-|z|^2/(C\rho)},
 \qquad
 |\nabla H(\rho,z)|
 \le C|z|\rho^{-2}e^{-|z|^2/(C\rho)}
\]
give
\[
 |G_\epsilon(z)|
 \le C\int_0^\infty \rho^{-3/2}e^{-|z|^2/(C\rho)}\dd\rho
 \le C|z|^{-1}
\]
and
\[
 |K_\epsilon(z)|
 \le C|z|\int_0^\infty
 \rho^{-5/2}e^{-|z|^2/(C\rho)}\dd\rho
 \le C|z|^{-2}.
\]
The mean-value theorem now proves
\eqref{eq:regularized-whole-space-paired}.  Finally,
$G_0-G_\epsilon$ and $K_0-K_\epsilon$ are given by the same integrals over
$0<\rho<\epsilon$.  When $|z|\ge\eta$, their integrands are bounded by an
integrable power of $\rho$ times $e^{-c\eta^2/\rho}$, uniformly in $z$
(polynomial factors in $|z|/\sqrt\rho$ are absorbed into a slightly weaker
Gaussian).  The resulting integrals tend to zero as $\epsilon\downarrow0$,
which proves \eqref{eq:whole-space-off-diagonal-convergence}.
\end{proof}

\subsection{Bounded domains}

We assume throughout this subsection that $\Om\subset\R^2$ is a bounded domain
with $C^5$ boundary, and let $H_D(t,x,y)$ denote its Dirichlet heat kernel.
For $\epsilon\ge0$ we set
\begin{equation}\label{eq:regularized-domain-kernels}
  G_{1,\epsilon}(x,y)
  =c_0\int_\epsilon^\infty
  H_D(t,x,y)t^{-1/2}\dd t,
  \qquad
  K_\epsilon(x,y)=\nabla_x^\perp G_{1,\epsilon}(x,y).
\end{equation}
Self-adjointness of the Dirichlet heat semigroup gives
$H_D(t,x,y)=H_D(t,y,x)$ and hence
$G_{1,\epsilon}(x,y)=G_{1,\epsilon}(y,x)$.  In particular,
$G_{1,0}=G_1$ is the kernel of $\Lam^{-1}$; we write $K=K_0$.
For smooth $\theta$, the velocity therefore has the representation
\[
  u(x)=\int_\Om K(x,y)\theta(y)\dd y.
\]
If $\varphi\in C_c^\infty(\Om)$, incompressibility and tangency of $u$ to
$\partial\Om$ give
\[
  \int_\Om u\cdot\nabla\theta\,\varphi\dd x
  =-
  \iint_{\Om\times\Om}K(x,y)\cdot\nabla\varphi(x)
  \theta(x)\theta(y)\dd x\dd y.
\]
Unlike the whole-space kernel, $K(x,y)$ is not translation invariant and is
not simply antisymmetric under exchange of its arguments.  The correct 
replacement for oddness is obtained by averaging the expression with its
$x$--$y$ transpose, which gives
\begin{equation}\label{eq:domain-symmetrized}
  \int_\Om u\cdot\nabla\theta\,\varphi\dd x
  =-\frac12
  \iint_{\Om\times\Om}
  \left[K(x,y)\cdot\nabla\varphi(x)+K(y,x)\cdot\nabla\varphi(y)\right]
  \theta(x)\theta(y)\dd x\dd y.
\end{equation}

We write
\[
  \delta(x)=\dist(x,\partial\Omega).
\]
Gaussian and gradient estimates (see, for example,
\cite{ConstantinIgnatova2017,ConstantinNguyen2018}) state that, for
$0<t\leq t_0$,
\begin{align}
  0\leq H_D(t,x,y)
  &\leq Ct^{-1}
  \min\left\{1,\frac{\delta(x)}{\sqrt t}\right\}
  \min\left\{1,\frac{\delta(y)}{\sqrt t}\right\}
  \exp\left(-\frac{|x-y|^2}{Ct}\right),
  \label{eq:dirichlet-heat-bound}\\
  |\nabla_xH_D(t,x,y)|+|\nabla_yH_D(t,x,y)|
  &\leq Ct^{-3/2}
  \exp\left(-\frac{|x-y|^2}{Ct}\right).
  \label{eq:dirichlet-heat-gradient-bound}
\end{align}
We also use the short-time estimate
\begin{equation}\label{eq:dirichlet-simultaneous-gradient}
 \left|(\nabla_x+\nabla_y)H_D(t,x,y)\right|
 \leq Ct^{-3/2}
 \exp\left(-\frac{\delta(x)^2}{Kt}\right),
 \qquad 0<t\leq c\delta(x)^2.
\end{equation}
This follows from the pointwise estimates
\cite[(24)--(27)]{ConstantinIgnatova2020}.  By symmetry, the
same estimate holds with $\delta(y)$ in place of $\delta(x)$.
For $t\geq t_0$, the kernel and its first derivatives decay exponentially
by the spectral expansion.  Integrating these estimates in
\eqref{eq:regularized-domain-kernels} gives the global kernel bounds needed
below.  Recall that $\epsilon$ is a heat-time parameter, so its spatial scale
is $\sqrt\epsilon$.

\begin{corollary}
\label{cor:regularized-dirichlet-kernel-bounds}
Let $r=|x-y|$.  Uniformly for $\epsilon\geq0$ and $x\neq y$ in $\Om$,
\begin{align}
 |G_{1,\epsilon}(x,y)|
 &\leq \frac{C_\Om}{r+\sqrt\epsilon},
 \label{eq:regularized-domain-G-bound}\\
 |\nabla_xG_{1,\epsilon}(x,y)|
 +|\nabla_yG_{1,\epsilon}(x,y)|
 &\leq \frac{C_\Om}{(r+\sqrt\epsilon)^2}.
 \label{eq:regularized-domain-gradient-bound}
\end{align}
Moreover, for every $\eta>0$,
\begin{equation}\label{eq:domain-off-diagonal-convergence}
\begin{aligned}
 &\sup_{\substack{x,y\in\Om\\ |x-y|\geq\eta}}
 \bigl(
 |G_{1,\epsilon}(x,y)-G_1(x,y)|
 +|\nabla_x(G_{1,\epsilon}-G_1)(x,y)|\\
 &\hspace{46mm}
 +|\nabla_y(G_{1,\epsilon}-G_1)(x,y)|
 \bigr)
 \longrightarrow0
\end{aligned}
\end{equation}
as $\epsilon\downarrow0$.
For every $\eta>0$ and $\varphi\in W^{2,\infty}(\Om)$, one also has
\begin{equation}\label{eq:domain-kernel-bound-interior}
 \left|K_\epsilon(x,y)\cdot\nabla\varphi(x)
 +K_\epsilon(y,x)\cdot\nabla\varphi(y)\right|
 \leq
 \frac{C_{\Om,\eta}\|\varphi\|_{W^{2,\infty}(\Om)}}{|x-y|}
\end{equation}
whenever $x\ne y$ and $\delta(x)+\delta(y)>\eta$, uniformly in
$\epsilon\geq0$.
\end{corollary}

\begin{proof}
Gaussian integration gives
\begin{align*}
 \int_\epsilon^{t_0}t^{-3/2}
       e^{-r^2/(Ct)}\dd t
 &\leq \frac{C}{r+\sqrt\epsilon},\\
 \int_\epsilon^{t_0}t^{-2}
       e^{-r^2/(Ct)}\dd t
 &\leq \frac{C}{(r+\sqrt\epsilon)^2}.
\end{align*}
The large-time parts are smooth and exponentially decreasing.  Integrating
\eqref{eq:dirichlet-heat-bound} and
\eqref{eq:dirichlet-heat-gradient-bound} proves
\eqref{eq:regularized-domain-G-bound} and
\eqref{eq:regularized-domain-gradient-bound}.  Each difference in
\eqref{eq:domain-off-diagonal-convergence} is an integral over
$0<t<\epsilon$.  On $r\geq\eta$, its integrand is bounded by a power of
$t^{-1}$ times $\exp(-c\eta^2/t)$, so uniform dominated convergence proves
the convergence assertion.

It remains to prove \eqref{eq:domain-kernel-bound-interior}.  We set
$r=|x-y|$.  If $r\geq\eta/4$, then
\eqref{eq:regularized-domain-gradient-bound} gives
\[
 \left|K_\epsilon(x,y)\cdot\nabla\varphi(x)
 +K_\epsilon(y,x)\cdot\nabla\varphi(y)\right|
 \leq \frac{C_{\Om,\eta}}r\|\nabla\varphi\|_{L^\infty}.
\]
Suppose now that $r<\eta/4$.  Since the distance function is
$1$-Lipschitz,
\[
 \min\{\delta(x),\delta(y)\}
 =\frac{\delta(x)+\delta(y)-|\delta(x)-\delta(y)|}{2}
 >\frac{\eta-r}{2}>\frac{3\eta}{8}.
\]
The segment joining $x$ and $y$ lies in $\Om$, and hence
\[
 |\nabla\varphi(x)-\nabla\varphi(y)|
 \leq r\|\nabla^2\varphi\|_{L^\infty}.
\]
Using the symmetry of $H_D$, we write the expression on the left of
\eqref{eq:domain-kernel-bound-interior} as
\begin{align*}
 c_0\int_\epsilon^\infty
 &\left\{\nabla_x^\perp H_D(t,x,y)\cdot
       \bigl(\nabla\varphi(x)-\nabla\varphi(y)\bigr)\right.\\
 &\left.\qquad
 +(\nabla_x^\perp+\nabla_y^\perp)H_D(t,x,y)
       \cdot\nabla\varphi(y)\right\}t^{-1/2}\dd t.
\end{align*}
The first term is bounded by $C r^{-1}\|\nabla^2\varphi\|_{L^\infty}$
using \eqref{eq:dirichlet-heat-gradient-bound}.  For the second term,
\eqref{eq:dirichlet-simultaneous-gradient} controls the short-time integral,
while the ordinary gradient bound and the spectral decay control the
remaining times.  Since $\delta(x),\delta(y)>3\eta/8$, this gives
\[
 C_{\Om,\eta}\|\nabla\varphi\|_{L^\infty}
 \leq \frac{C_{\Om,\eta}}r\|\nabla\varphi\|_{L^\infty},
\]
after changing the constant, since $r\leq\operatorname{diam}(\Om)$.  This proves
\eqref{eq:domain-kernel-bound-interior}.
\end{proof}

\begin{corollary}
\label{lem:uniform-paired-dirichlet-kernel}
Let $\varphi\in C_c^\infty(\Om)$.  Then, uniformly for $\epsilon\geq0$ and
$x\neq y$ in $\Om$,
\begin{equation}\label{eq:domain-kernel-bound}
 \left|K_\epsilon(x,y)\cdot\nabla\varphi(x)
 +K_\epsilon(y,x)\cdot\nabla\varphi(y)\right|
 \leq \frac{C_{\Om,\varphi}}{|x-y|}.
\end{equation}
\end{corollary}

\begin{proof}
If $\nabla\varphi\equiv0$, the assertion is immediate.  Otherwise set
\[
 d_\varphi=\dist(\supp\nabla\varphi,\partial\Om)>0,
 \qquad \eta_\varphi=\frac{d_\varphi}{2}.
\]
If $\delta(x)+\delta(y)\leq\eta_\varphi$, then
$\nabla\varphi(x)=\nabla\varphi(y)=0$, so the left-hand side vanishes.  If
$\delta(x)+\delta(y)>\eta_\varphi$, apply
Corollary~\ref{cor:regularized-dirichlet-kernel-bounds}.  This proves
\eqref{eq:domain-kernel-bound}, with
$C_{\Om,\varphi}=C_{\Om,\eta_\varphi}
\|\varphi\|_{W^{2,\infty}}$.
\end{proof}

Taking $\epsilon=0$ in
Corollary~\ref{lem:uniform-paired-dirichlet-kernel}
proves the asserted bound for \eqref{eq:domain-symmetrized}.  Consequently,
\begin{equation}\label{eq:time-bound-domain}
  \left|\int_\Om u\cdot\nabla\theta\,\varphi\dd x\right|
  \lesssim_{\Om,\varphi}
  \|\theta\|_{L^{4/3,2}}^2.
\end{equation}
Thus both geometries give the same critical estimate for each fixed compactly
supported test function.  For $\epsilon\geq0$, define
\[
 A_{\varphi,\epsilon}(x,y)
 :=\frac12
 \begin{cases}
 K_\epsilon(x-y)\cdot
 \bigl(\nabla\varphi(x)-\nabla\varphi(y)\bigr),
 &\Om=\R^2,\\[1mm]
 K_\epsilon(x,y)\cdot\nabla\varphi(x)
 +K_\epsilon(y,x)\cdot\nabla\varphi(y),
 &\Om\text{ bounded},
 \end{cases}
\]
and let
\begin{equation}\label{eq:regularized-symmetrized-form}
 B_{\varphi,\epsilon}(f,g)
 :=\iint_{\Om\times\Om}
 A_{\varphi,\epsilon}(x,y)f(x)g(y)\dd x\dd y.
\end{equation}
The kernel $A_{\varphi,\epsilon}$ is symmetric in $(x,y)$, and the estimates
proved above give, uniformly for $0\leq\epsilon\leq1$,
\begin{equation}\label{eq:A-bound}
 |A_{\varphi,\epsilon}(x,y)|
 \leq \frac{C_{\Om,\varphi}}{|x-y|},
 \qquad x\neq y.
\end{equation}
Moreover, $A_{\varphi,\epsilon}\to A_{\varphi,0}$ uniformly on every set
$\{|x-y|\geq\eta\}$, after the natural spatial localization in the
whole-space case.  Proposition~\ref{prop:critical-lorentz} therefore shows
that $B_{\varphi,\epsilon}$ is a bounded symmetric bilinear form on
$X\times X$, with norm independent of $\epsilon$.

At $\epsilon=0$, \eqref{eq:regularized-symmetrized-form} agrees with the
transport form introduced in \eqref{eq:intro-transport-form}:
\[
  B_{\varphi,0}(\theta,\theta)=\mathcal N_\Om(\theta,\varphi),
  \qquad
  \langle u\cdot\nabla\theta,\varphi\rangle
  =-B_{\varphi,0}(\theta,\theta).
\]
For time-dependent tests these identities are applied at almost every time
and then integrated.  Bounds \eqref{eq:time-bound-whole} and
\eqref{eq:time-bound-domain} ensure the required local integrability.

\section{Compactness and proof of the main theorem}
\label{sec:compactness-proof}

We now pass to the limit in the approximate solutions.  The symmetrized
estimate gives time compactness.  We then split the quadratic kernels into
near- and far-diagonal parts.  The tail estimate controls the near part, and
the same argument applies to the Hamiltonian.

Fix a sequence
\[
  (\epsilon_{1,n},\epsilon_{2,n})\longrightarrow(0,0),
  \qquad
  \theta_n:=\theta^{\epsilon_{1,n},\epsilon_{2,n}}.
\]
All estimates below are uniform on a fixed interval $[0,T]$, and a diagonal
extraction in $T$ is used for $T\to\infty$.

\subsection{Compactness and inherited localization}

From \eqref{eq:time-bound-whole}, \eqref{eq:time-bound-domain}, and the uniform
bound \eqref{eq:uniform-lorentz}, every fixed $\zeta\in C_c^\infty(\Om)$
satisfies
\begin{equation}\label{eq:time-compactness}
 \left|\frac{\dd}{\dd t}
 \langle\theta_n(t),\zeta\rangle\right|
 \le C_\zeta
 \qquad\text{for a.e. }t,
\end{equation}
uniformly in $n$.  Hence the scalar pairings are equibounded and
equi-Lipschitz.  Arzel\`a--Ascoli, applied to a countable dense family of test
functions, and a diagonal extraction give
\begin{equation}\label{eq:compactness-convergences}
 \theta_n\rightharpoonup^\ast\theta
 \quad\text{in }L^\infty(0,T;X),
 \qquad
 \theta_n\longrightarrow\theta
 \quad\text{in }C([0,T];\mathcal D'(\Om)).
\end{equation}
Since $X$ is reflexive, convergence in distributions at a fixed time,
together with the uniform $X$ bound, implies
\[
  \theta_n(t)\rightharpoonup\theta(t)
  \qquad\text{weakly in }X
\]
for every $t\in[0,T]$.  In fact, this convergence is uniform in the weak
topology.  Indeed, $C_c^\infty(\Om)$ is dense in
$X^*=L^{4,2}(\Om)$, and therefore \eqref{eq:compactness-convergences} and the
uniform $X$ bound imply
\begin{equation}\label{eq:weak-X-uniform}
  \sup_{0\leq t\leq T}
  |\langle\theta_n(t)-\theta(t),g\rangle|
  \longrightarrow0
  \qquad\text{for every }g\in X^*.
\end{equation}
Because $e^{\epsilon_{2,n}\Delta}\theta_0\to\theta_0$ in $X$, the
convergence at $t=0$ gives the prescribed initial trace
$\theta(0)=\theta_0$.

Weak lower semicontinuity of the convex high amplitude cutoff modulus and
\eqref{eq:uniform-tail} yield, for every $M\geq0$ and $t\in[0,T]$,
\begin{equation}\label{eq:limit-tail-inheritance}
  \tau_M(\theta(t))
  \leq\liminf_{n\to\infty}\tau_M(\theta_n(t))
  \leq\tau_M(\theta_0).
\end{equation}
This proves \eqref{eq:tail-bound}.  Applying Proposition
\ref{prop:weighted-tail-envelope} with $f=\theta_0$ and $h=\theta(t)$ gives
\eqref{eq:weighted-distribution-bound} with one weight
$\Phi_{\theta_0}$, independent of time.  This step does not require weak lower semicontinuity of 
the weighted functional.

We also record the spatial localization inherited by the limit.  If
$\Om=\R^2$, Lemma~\ref{lem:whole-space-tightness} and weak lower
semicontinuity imply
\begin{equation}\label{eq:limit-spatial-tightness}
 \lim_{R\to\infty}
 \left[
  \sup_n\sup_{0\leq t\leq T}
  \|\theta_n(t)\one_{B_R^c}\|_X
  +\sup_{0\leq t\leq T}
  \|\theta(t)\one_{B_R^c}\|_X
 \right]=0.
\end{equation}
If $\Om$ is bounded, let $E_\rho=\{x\in\Om:\delta(x)<\rho\}$.  For
$f=\theta_n(t)$ or $f=\theta(t)$, the decomposition
$f=T_Mf+R_Mf$ gives
\begin{equation}\label{eq:boundary-layer-tightness}
  \|f\one_{E_\rho}\|_X
  \leq C M|E_\rho|^{3/4}+\tau_M(\theta_0).
\end{equation}
Since $|E_\rho|\to0$, choosing first $M$ and then $\rho$ makes the
right-hand side arbitrarily small, uniformly in $n$ and $t\in[0,T]$.

\subsection{Compactness of the quadratic forms}

We use the following lemma for both the transport term and the Hamiltonian.

\begin{lemma}
\label{lem:critical-quadratic-compactness}
Let $f_n,f\in L^\infty(0,T;X)$ satisfy
\[
  f_n\longrightarrow f\quad\text{in }C([0,T];X_{\rm w}),
  \qquad
  \sup_n\sup_{0\leq t\leq T}\|f_n(t)\|_X
  +\sup_{0\leq t\leq T}\|f(t)\|_X\leq A,
\]
and suppose that, for some function $\omega(M)\downarrow0$,
\[
  \sup_n\sup_{0\leq t\leq T}\tau_M(f_n(t))
  +\sup_{0\leq t\leq T}\tau_M(f(t))\leq\omega(M).
\]
Fix a compact exhaustion $K_1\Subset K_2\Subset\cdots\Subset\Om$ with
$\bigcup_mK_m=\Om$; when $\Om=\R^2$, we take $K_m=\overline{B_m}$.
Assume in addition the spatial localization
\begin{equation}\label{eq:abstract-spatial-localization}
 \lim_{m\to\infty}
 \left[
  \sup_n\sup_{0\leq t\leq T}\|f_n(t)\one_{\Om\setminus K_m}\|_X
  +\sup_{0\leq t\leq T}\|f(t)\one_{\Om\setminus K_m}\|_X
 \right]=0.
\end{equation}

Let $J_n,J$ be symmetric kernels that are jointly continuous in
$(t,x,y)$ away from $x=y$,
such that
\begin{equation}\label{eq:abstract-critical-kernel-bound}
  |J_n(t,x,y)|+|J(t,x,y)|\leq\frac{C}{|x-y|},
\end{equation}
and, for every $K\Subset\Om$ and $\eta>0$,
\[
  J_n\longrightarrow J
  \quad\text{uniformly on }
  [0,T]\times\{(x,y)\in K^2:|x-y|\geq\eta\}.
\]
Then
\begin{equation}\label{eq:abstract-quadratic-convergence}
 \sup_{0\leq t\leq T}
 \left|
  \iint J_n(t,x,y)f_n(t,x)f_n(t,y)\dd x\dd y
  -\iint J(t,x,y)f(t,x)f(t,y)\dd x\dd y
 \right|\longrightarrow0.
\end{equation}
\end{lemma}

\begin{proof}
We choose $\chi\in C_c^\infty([0,\infty))$ with $\chi=1$ on $[0,1]$ and
$\chi=0$ on $[2,\infty)$.  We split the kernels using
$\chi(|x-y|/\eta)$ and write
\begin{align*}
 Q_{L,\eta}^{\rm near}(h)
 &:=\iint \chi\left(\frac{|x-y|}{\eta}\right)
 L(t,x,y)h(x)h(y)\dd x\dd y,\\
 Q_{L,\eta}^{\rm far}(h)
 &:=\iint \left[1-\chi\left(\frac{|x-y|}{\eta}\right)\right]
 L(t,x,y)h(x)h(y)\dd x\dd y.
\end{align*}
We begin with the far part.  If $L$ satisfies
\eqref{eq:abstract-critical-kernel-bound}, then Proposition
\ref{prop:critical-lorentz} gives
\[
 \left|Q_{L,\eta}^{\rm far}(h)
       -Q_{L,\eta}^{\rm far}(h\one_K)\right|
 \leq C\|h\|_X\|h\one_{\Om\setminus K}\|_X.
\]
By \eqref{eq:abstract-spatial-localization}, we may therefore choose
$K=K_m$ so that the contribution from $\Om\setminus K$ is arbitrarily small,
uniformly in $n$ and $t$.

On $K$, the embedding $X\hookrightarrow L^1(K)$ gives a uniform bound for
$\|f_n(t)\|_{L^1(K)}$ and $\|f(t)\|_{L^1(K)}$.  The off-diagonal convergence
of $J_n$ then gives
\[
 \sup_{0\leq t\leq T}
 \left|Q_{J_n,\eta}^{\rm far}(f_n(t)\one_K)
       -Q_{J,\eta}^{\rm far}(f_n(t)\one_K)\right|
 \longrightarrow0.
\]
The far kernel is continuous on $[0,T]\times K\times K$.  The standard
compact-kernel argument, the convergence in $C([0,T];X_{\rm w})$, and the
uniform $L^1(K)$ bound therefore give
\begin{equation}\label{eq:far-field-convergence}
 \sup_{0\leq t\leq T}
\left|Q_{J_n,\eta}^{\rm far}(f_n(t))
       -Q_{J,\eta}^{\rm far}(f(t))\right|\longrightarrow0,
\end{equation}

It remains to control the diagonal uniformly.  For any member $h$ of the
family $\{f_n(t),f(t)\}$, write
\[
  |h|=h_{\leq M}+h_{>M},
  \qquad
  h_{\leq M}=|h|\wedge M,
  \qquad
  h_{>M}=(|h|-M)_+.
\]
The embedding $X\hookrightarrow L^{4/3,\infty}$ 
gives
\[
  \|h_{\leq M}\|_{L^2}^2
  \leq C M^{2/3}A^{4/3}.
\]
Young's inequality and
$\||z|^{-1}\one_{\{|z|\leq2\eta\}}\|_{L^1(\R^2)}\leq C\eta$ therefore
imply
\begin{equation}\label{eq:near-bounded-part}
 \iint_{|x-y|\leq2\eta}
 \frac{h_{\leq M}(x)h_{\leq M}(y)}{|x-y|}\dd x\dd y
 \leq C M^{2/3}A^{4/3}\eta.
\end{equation}
The boundedness $\Lambda^{-\frac{1}{2}}:X\to L^2$ gives
\begin{align}
 \iint\frac{h_{>M}(x)h_{>M}(y)}{|x-y|}\dd x\dd y
 &\leq C\omega(M)^2,
 \label{eq:near-high-part}\\
 \iint\frac{h_{\leq M}(x)h_{>M}(y)}{|x-y|}\dd x\dd y
 &\leq C A\omega(M).
 \label{eq:near-mixed-part}
\end{align}
Combining the three estimates and using
\eqref{eq:abstract-critical-kernel-bound}, we obtain for both the
approximating and limiting near-field forms
\begin{equation}\label{eq:near-small}
 \sup_n\sup_{0\leq t\leq T}|Q_{J_n,\eta}^{\rm near}(f_n(t))|
 +\sup_{0\leq t\leq T}|Q_{J,\eta}^{\rm near}(f(t))|
 \leq C\bigl(M^{2/3}A^{4/3}\eta+A\omega(M)+\omega(M)^2\bigr).
\end{equation}
We first choose $M$ so that the last two terms are small and then choose $\eta$
so that the first term is small.  Together with
\eqref{eq:far-field-convergence}, this proves
\eqref{eq:abstract-quadratic-convergence}.
\end{proof}

We apply the lemma to the transport term.  Let
$\varphi\in C_c^\infty([0,T]\times\Om)$ and set
\[
  J_n(t,x,y)=A_{\varphi(t),\epsilon_{1,n}}(x,y),
  \qquad
  J(t,x,y)=A_{\varphi(t),0}(x,y).
\]
The uniform bound is \eqref{eq:A-bound}; off-diagonal convergence follows
from Lemma~\ref{lem:regularized-whole-space-kernel} on $\R^2$ and from
Corollary~\ref{cor:regularized-dirichlet-kernel-bounds} on a bounded domain.
The high amplitude cutoff  hypothesis follows from \eqref{eq:uniform-tail} and
\eqref{eq:limit-tail-inheritance}, while the localization hypothesis is
\eqref{eq:limit-spatial-tightness} or
\eqref{eq:boundary-layer-tightness}.  Hence
\begin{equation}\label{eq:nonlinear-convergence}
 \sup_{0\leq t\leq T}
 \left|
 B_{\varphi(t),\epsilon_{1,n}}(\theta_n(t),\theta_n(t))
 -B_{\varphi(t),0}(\theta(t),\theta(t))
 \right|\longrightarrow0.
\end{equation}
Passing to the limit in the approximate weak formulation proves
\eqref{eq:weak-solution-definition}.  Thus $\theta$ is a global weak solution
with initial datum $\theta_0$.

\subsection{Compactness of the negative Sobolev Hamiltonian}

Weak convergence alone would give only lower semicontinuity of the
Hamiltonian.  The preceding quadratic-form lemma excludes an energy defect.
Let $\mathscr G_\epsilon$ denote $G_\epsilon(x-y)$ on $\R^2$ and
$G_{1,\epsilon}(x,y)$ on a bounded domain.  In either case,
$\mathscr G_\epsilon$ is symmetric and
\[
  |\mathscr G_\epsilon(x,y)|\leq\frac{C_\Om}{|x-y|}
\]
uniformly for $\epsilon\geq0$.  Moreover,
$\mathscr G_{\epsilon_{1,n}}\to\mathscr G_0$ uniformly off the diagonal on
compact subsets, by Lemma~\ref{lem:regularized-whole-space-kernel} or
Corollary~\ref{cor:regularized-dirichlet-kernel-bounds}.

We apply Lemma~\ref{lem:critical-quadratic-compactness} with
\[
  J_n=\mathscr G_{\epsilon_{1,n}},
  \qquad J=\mathscr G_0,
  \qquad f_n=\theta_n,
  \qquad f=\theta.
\]
All high amplitude cutoff  and localization assumptions were established in the preceding
subsection.  We obtain, for every $T>0$,
\begin{equation}\label{eq:energy-convergence}
 \lim_{n\to\infty}\sup_{0\leq t\leq T}
 \left|
  \int_\Om\theta_n(t)\Lam_{\epsilon_{1,n}}^{-1}\theta_n(t)\dd x
  -\int_\Om\theta(t)\Lam^{-1}\theta(t)\dd x
 \right|=0.
\end{equation}
\subsection{Conservation of the negative Sobolev Hamiltonian}

The regularized energy identity \eqref{eq:regularized-energy} gives
\[
  \int_\Om \theta_n(t)
  \Lam_{\epsilon_{1,n}}^{-1}\theta_n(t)\dd x
  =
  \int_\Om e^{\epsilon_{2,n}\Delta}\theta_0
  \Lam_{\epsilon_{1,n}}^{-1}
  e^{\epsilon_{2,n}\Delta}\theta_0\dd x.
\]
By \eqref{eq:initial-regularized-energy-convergence}, the right-hand side
converges to $\int_\Om\theta_0\Lam^{-1}\theta_0\dd x$.
On the other hand, \eqref{eq:energy-convergence} identifies the limit of the
left-hand side, uniformly for $t\in[0,T]$.  Therefore
\[
  \int_\Om\theta(t)\Lam^{-1}\theta(t)\dd x
  =\int_\Om\theta_0\Lam^{-1}\theta_0\dd x
  \qquad (0\leq t\leq T).
\]
Since $T>0$ is arbitrary, this identity holds for every $t\geq0$.  Equivalently,
\[
  \|\theta(t)\|_{\dot H^{-1/2}(\Om)}
  =\|\theta_0\|_{\dot H^{-1/2}(\Om)},
  \qquad t\geq0.
\]

This completes the proof of Theorem~\ref{thm:main}.

\subsection*{Acknowledgments} P.C was partially supported by NSF grant DMS-2106528
and by a Simons Collaboration Grant 601960. The work of M.I. was partially supported by NSF grant DMS-2204614.
Q.-H.~N. was supported by the CAS Project for Young Scientists in Basic
Research, Grant No.~YSBR-031, and by the NSFC under Grant
Nos.~1251101538 and 12595282.

\appendix

\section{Optimality of the secondary Lorentz exponent}\label{app:sharpness}

We show that the secondary Lorentz index in
Proposition~\ref{prop:critical-lorentz} cannot be increased while the principal
exponent remains $4/3$.  The examples are radial, nonnegative, and supported
near the origin.  We first consider $q=\infty$ and then add logarithmic factors
for $2<q<\infty$.

\begin{proof}[Proof of the sharpness statement in
Proposition~\ref{prop:critical-lorentz}]
For $q=\infty$, take
\[
  f(x)=\one_{\{|x|<1/100\}}|x|^{-3/2}.
\]
Then $f\in L^{4/3,\infty}(\R^2)$, whereas the annulus estimate used below
gives $I_{1/2}f(x)\gtrsim |x|^{-1}$ near the origin, and hence
$I_{1/2}f\notin L^2$.  It remains to treat $2<q<\infty$.
	
\begingroup
We fix $2<q<\infty$ and choose $a>1/q$.  We write
$L(r)=\log(1/r)$ and choose $r_0\in(0,e^{-e})$ sufficiently small that the
radial function
\[
 f(x)=\one_{\{|x|<r_0\}}|x|^{-3/2}
       L(|x|)^{-1/q}\bigl(\log L(|x|)\bigr)^{-a}
\]
is radially decreasing on its support.  For a radial decreasing function in
two dimensions, the rearrangement formula for Lorentz norms gives
\begin{align*}
 \|f\|_{L^{4/3,q}}^q
 &\simeq
 \int_0^{r_0}\bigl[r^{3/2}f(r)\bigr]^q\frac{\dd r}{r}=\int_{L(r_0)}^\infty
   \frac{\dd L}{L(\log L)^{aq}}<\infty,
\end{align*}
because $aq>1$.  Thus $f\in L^{4/3,q}(\R^2)$.

We next show that $I_{1/2}f\notin L^2$.  Let $|x|=r<r_0/4$ and restrict the
potential integral to $2r<|y|<4r$.  On this annulus,
$|x-y|\simeq r$, and the explicit radial profile gives
\[
 f(y)\simeq r^{-3/2}L(r)^{-1/q}(\log L(r))^{-a}.
\]
Since the annulus has area comparable to $r^2$ and $f\geq0$, it follows that
\[
 I_{1/2}f(x)
 =c_2\int_{\R^2}\frac{f(y)}{|x-y|^{3/2}}\dd y
 \gtrsim r^{-1}L(r)^{-1/q}(\log L(r))^{-a}.
\]
Consequently,
\begin{align*}
 \|I_{1/2}f\|_{L^2}^2
 &\gtrsim
 \int_0^{r_0/4}
 \frac{\dd r}{rL(r)^{2/q}(\log L(r))^{2a}}=\int_{L(r_0/4)}^\infty
 \frac{\dd L}{L^{2/q}(\log L)^{2a}}
 =\infty,
\end{align*}
because $2/q<1$.  Hence $f\in L^{4/3,q}(\R^2)$ but
$|D|^{-1/2}f\notin L^2(\R^2)$, and the asserted estimate cannot hold for any
$q>2$.
\endgroup
\end{proof}

\end{document}